\documentclass[11pt,a4paper,reqno]{amsart}
\usepackage{amssymb,latexsym,amsmath,amscd,amsthm,amsfonts, enumerate}
\usepackage{multirow}
\usepackage{color}
\usepackage[all]{xy}
\usepackage{etex}
\usepackage{caption}
\usepackage{soul}
\usepackage{float}
\usepackage{titletoc}
\usepackage[normalem]{ulem}
\usepackage{graphicx}
\usepackage[usenames,dvipsnames]{xcolor}
\usepackage{url}

\usepackage{tikz,tikz-cd}
\usepackage{mathrsfs}
\usepackage[all]{xy}
\usepackage{tikz}
\usepackage{extarrows}
\usepackage{tikz-cd}
\usetikzlibrary{calc}
\usetikzlibrary{matrix,arrows,decorations.pathmorphing}
\usetikzlibrary{snakes}
\usetikzlibrary{shapes.geometric,positioning}
\usetikzlibrary{arrows,decorations.pathmorphing,decorations.pathreplacing}
\usetikzlibrary{positioning,shapes,shadows,arrows,snakes}

\usepackage[colorlinks=true,pagebackref,hyperindex]{hyperref}
\newcommand\myshade{85}
\colorlet{mylinkcolor}{violet}
\colorlet{mycitecolor}{red}
\colorlet{myurlcolor}{cyan}

\hypersetup{
  linkcolor  = mylinkcolor!\myshade!black,
  citecolor  = mycitecolor!\myshade!black,
  urlcolor   = myurlcolor!\myshade!black,
  colorlinks = true,
}

\numberwithin{equation}{section}

\newtheorem{theorem}{Theorem}[section]
\newtheorem{proposition}[theorem]{Proposition}
\newtheorem{proposition-definition}[theorem]{Proposition-Definition}
\newtheorem{conjecture}[theorem]{Conjecture}
\newtheorem{question}[theorem]{Question}
\newtheorem{corollary}[theorem]{Corollary}
\newtheorem{lemma}[theorem]{Lemma}

\theoremstyle{definition}
\newtheorem{remark}[theorem]{Remark}

\newtheorem{definition}[theorem]{Definition}

\newcommand{\exc}{\operatorname{exp}\nolimits}
\newcommand{\End}{\operatorname{End}\nolimits}
\newcommand{\thick}{\mathsf{thick}}
\newcommand{\Hom}{\mathrm{Hom}}

\newcommand{\Int}{\mathrm{Int}}

\newcommand{\za}{\alpha}
\newcommand{\zb}{\beta}

\newcommand{\zg}{\gamma}

\newcommand{\Z}{\mathbb{Z}}

\newcommand{\calt}{\mathcal{T}}

\def\s{\stackrel}

\definecolor{dark-green}{RGB}{14,150,2}
\definecolor{red}{RGB}{250,0,0}

\begin{document}

\title[Braid group actions on branched coverings and full exceptional sequences]{Braid group actions on branched coverings and full exceptional sequences}

\author{Wen Chang}
\address{(W. Chang) School of Mathematics and Statistics, Shaanxi Normal University, Xi'an 710062, China}
\email{changwen161@163.com}

\author{Fabian Haiden}
\address{(F. Haiden) Centre for Quantum Mathematics, Department of Mathematics and Computer Science, University of Southern Denmark, Campusvej 55, 5230 Odense, Denmark}
\email{fab@sdu.dk}

\author{Sibylle Schroll}
\address{(S. Schroll) Insitut f\"ur Mathematik, Universit\"at zu K\"oln, Weyertal 86-90, K\"oln, Germany and
Institutt for matematiske fag, NTNU, N-7491 Trondheim, Norway}
\email{schroll@math.uni-koeln.de}

\keywords{branched covering, Hurwitz system, Fukaya category, braid group action, exceptional sequence, gentle algebra.}


\date{\today}

\subjclass[2010]{16E35, 
57M50}

\begin{abstract}
We relate full exceptional sequences in Fukaya categories of surfaces or equivalently in derived categories of graded gentle algebras to branched coverings over the disk, building on a previous classification result of the first and third author \cite{CS22}.
This allows us to apply tools from the theory of branched coverings  such as Birman--Hilden theory and Hurwitz systems to study the natural braid group action on exceptional sequences.
As an application, counterexamples are given to a conjecture of Bondal--Polishchuk~\cite{BP94} on the transitivity of the braid group action on full exceptional sequences in a triangulated category.
\end{abstract}

\maketitle
\setcounter{tocdepth}{1}

\tableofcontents

\section{Introduction}\label{Introductions}

Beilinson's discovery of a full exceptional sequence in $D^b(\mathrm{Coh}(\mathbb P^n))$, together with generalizations to other Fano varieties, established a bridge between algebraic geometry and representation theory of finite-dimensional algebras.
Bondal and Gorodentsev later independently showed that the set of full exceptional sequences in any triangulated category carries a natural action of the braid group $\mathfrak{B}_n$.
This action was studied further by Bondal and Polishchuk in~\cite{BP94} who made the following conjecture.

\begin{conjecture}[Bondal--Polishchuk]
\label{conj:bp}
In any triangulated category admitting a full exceptional sequence of length $n$, the group $\mathbb{Z}^n\rtimes \mathfrak{B}_n$ acts transitively on the set of full exceptional sequences.
\end{conjecture}

Here the group $\mathbb{Z}^n$ acts on full exceptional sequences by shifting the objects.
Equivalently, one can ask if $\mathfrak{B}_n$ acts transitively on full exceptional sequences \textit{up to shift}, i.e. modulo the action of $\mathbb{Z}^n$.
We will thus follow the convention that two exceptional sequences are considered the same if they only differ by shifts of the exceptional objects.

The conjecture has been verified for derived categories of hereditary algebras \cite{C93, R94}, for derived categories of coherent sheaves over some special varieties such as del Pezzo surfaces, projective planes, and weighted projective lines, \cite{KM02} see more details in \cite{M04}, as well as for the Hirzebruch surface of degree 2~\cite{IOU21}.
Recently, it was proved in work of the first and third author~\cite{CS22} that this conjecture holds for the derived category of a gentle algebra arising from a dissection of a marked surface with zero genus.
On the other hand, in this paper we will show here that the result does not extend to positive genus, providing counterexamples to Conjecture~\ref{conj:bp} --- the first ones found, to our knowledge. 

\begin{theorem}
\label{thm:intro1}
Let $(S,M,\nu)$  be a marked graded  surface where $S$ is a compact oriented surface with boundary, $M \subset \partial S$ a finite set of marked points with $|M|=2$ and $\nu$  a grading  (line field) on $S$. Suppose that $S$ has either one boundary component and genus $g(S)\geq 2$ or two boundary components and $g(S)\geq 1$.
Then the action of the braid group on the set of full exceptional sequences in the (partially wrapped) Fukaya category 
$\mathcal F(S,M,\nu)$
is not transitive.
\end{theorem}

See Theorem~\ref{thm:main1} in the main text.
Let us explain some of the ingredients of the proof.
Exceptional sequences and their mutation were given a geometric interpretation by Seidel in the theory of symplectic Lefschetz fibrations and the associated Fukaya--Seidel categories~\cite{S01a,S01b,S08}.
Very roughly speaking, symplectic Lefschetz fibrations $f:X\to D$ are Morse functions in the symplectic context.
Here, $X$ is symplectic (with boundary) and $D$ is the unit disk in $\mathbb C$.
One can associate a triangulated category $\mathrm{FS}(X,f)$ to them which comes, by construction, with a full exceptional sequence whose objects correspond to critical points of $f$.
The braid group acts as the mapping class group of $D\setminus \{\text{critical values}\}$.
In the case $\dim_{\mathbb R}(X)=2$, this reduces to the much more elementary theory of branched coverings of surfaces (specifically discussed in~\cite{S01b}).
Therefore, we will not assume here that the reader has any knowledge of the construction of Fukaya--Seidel categories, but instead describe the correspondence between branched covers of the disk and full exceptional sequences directly (of which only one direction is contained in the general higher-dimensional theory).

\begin{theorem}
\label{thm:intro2}
Let $S$ be an oriented surface with boundary and $M\subset \partial S$ a finite set of marked points, so that each component of $\partial S$ contains at least one point of $M$.
Then there are canonical bijections between equivalence classes of the following:
\begin{enumerate}
\item exceptional dissections of $(S,M)$ (in the sense of~\cite{CS22}),
\item simple branched coverings $\mathfrak p:S\to D$ of $S$ over the unit disk with $M=\mathfrak p^{-1}(-1)$ together with a choice of matching paths,and
\item full exceptional sequences in  $\mathcal F(S,M,\nu)$, 
for any choice of grading (line field) $\nu$.
\end{enumerate}
\end{theorem}

These terms are defined in detail in Section~\ref{sec:3sets}.
The above theorem is a combination of Theorem~\ref{main theorem bijections} and Theorem~\ref{thm:cor-exp} in the main text.
The question of transitivity of the braid group action is thus reduced to a question about branched coverings.
A useful tool here is Birman--Hilden theory~\cite{BH73,FM12,MW21}, which studies the relation between mapping class groups and braid groups.
In particular, we use a recent result of Ghaswala--McLeay~\cite{GM20} to complete the proof of Theorem~\ref{thm:intro1}.

A characterization of those marked surfaces $(S,M)$ for which the sets in Theorem~\ref{thm:intro2} are non-empty follows from~\cite{CS22}.

\begin{corollary}
Let $S$ be a compact oriented surface and $M\subset S$ a finite set of marked points, so that each component of $\partial S$ contains at least one point of $M$.
Then the following are equivalent:
\begin{enumerate}
\item There exists a simple branched covering $\mathfrak p:S\to D$ of $S$ over the unit disk $D$ with $M=\mathfrak p^{-1}(-1)$,
\item $(S,M)$ admits an exceptional dissection,
\item $\mathcal F(S,M,\nu)$ admits a full exceptional sequence (independently of the grading $\nu$),
\item $M\subset \partial S$ and either $|M|\geq 2$ or $|M|=1$ and $S=D$.
\end{enumerate}
\end{corollary}

The equivalence between (2), (3), and (4) is already contained in~\cite{CJS22} and \cite{CS22}, while the equivalence with (1) follows from Theorem~\ref{thm:intro2} above.

\subsection*{Acknowledgments}

This paper is partly a result of the ERC-SyG project Recursive and Exact New Quantum Theory (ReNewQuantum) which received funding from the European Research Council (ERC) under the European Union's Horizon 2020 research and innovation programme under grant agreement No 810573, 
the VILLUM FONDEN, VILLUM Investigator grant 37814, Sapere Aude grant 3120-00076B from the Independent Research Fund Denmark (DFF),
 by Fundamental Research Funds
 for the Central Universities (No. GK202403003), the NSF of China (Grant No. 12271321)
and the DFG through
the project SFB/TRR 191 Symplectic Structures in Geometry, Algebra and Dynamics (Projektnummer 281071066-TRR 191).

\section{Exceptional dissections, branched coverings, and Hurwitz systems}
\label{sec:3sets}

A \emph{marked surface} is a pair $(S,M)$, where $S$ is a compact oriented surface with boundary $\partial S\neq \emptyset$ and $M\subset \partial S$ is a finite subset such that each component of $\partial S$ contains at least one point of $M$.
A diffeomorphism of marked surfaces $f:(S,M)\rightarrow (S',M')$ is a diffeomorphism $f:S\rightarrow S'$ with $M'=f(M)$. {Denote by $\mathrm{Mod}(S)$ the mapping class group of $S$.}
For simplicity, we always assume that $S$ is connected.

Fix a marked surface $(S,M)$.
We will consider the following three sets:
\[\mathfrak{S}_1=\{\text{Isotopy classes of exceptional dissections on $(S,M)$}\},\]
\[\mathfrak{S}_2=\{\text{Equivalence classes of simple branched coverings $(S,M)\to D$ with matching paths}\},\]
\[{ \mathcal{H}=\{\text{Hurwitz systems of type $(S,M)$}\}.}\]

{ The terms used in the definitions of these sets will be defined precisely in this section. The main result is the following:}

\begin{theorem}\label{main theorem bijections}
There are canonical bijections {$\mathfrak S_1\cong\mathfrak S_2$ and $\mathfrak S_1/\mathrm{Mod}(S)\cong \mathcal{H}$}.
\end{theorem}

\subsection{Exceptional dissections}

\begin{definition}
Let $(S,M)$ be a marked surface. We call an ordered set $\mathbb{A}=(a_1,a_2,\ldots,a_n)$ of embedded arcs in $S$ with endpoints in $M$ an \emph{exceptional dissection}, if the arcs meet only in endpoints and cut $S$ into polygons, each of which has exactly one boundary edge, and if $a_j$ follows $a_i$ counterclockwise at a common endpoint, then $i<j$.
Two exceptional dissections are considered equivalent if they differ by isotopies of the arcs fixing the endpoints.
\end{definition}

For a marked surface $(S,M)$, we denote by $b$ the number of boundary components, by $m$ the number of marked points, and by $g$ the genus of $S$.
Then there are exactly
\[ n=m+b+2g-2 \]
arcs in an exceptional dissection on $(S,M)$, see for example \cite{APS19}, noticing that an exceptional dissection is a special kind of so called admissible dissection considered in \cite{APS19}.

\subsection{Branched coverings}

A smooth proper map  ${\mathfrak p}:S\to S'$ of surfaces is a \emph{branched covering} if it is a (finite sheeted) covering of topological spaces away from a discrete set of points $B\subset S'\setminus \partial S'$ (the \emph{branch points}) and over each point of $B$ one or more sheets of the covering come together in $S$ (at the \emph{ramification points}).
A branched covering is \emph{simple} if precisely two sheets come together over any branch point, i.e. all fibers of  $\mathfrak p$ have $m$ or $m-1$ points, where $m$ is the number of sheets of the covering.

Let $D:=\{z\in \mathbb{C}, |z|\leqslant 1\}$ be the unit disk and choose $-1\in \partial D$ as a basepoint.

\begin{definition}
Let $(S,M)$ be a marked surface.
A \emph{branched covering $(S,M)\to D$} is a simple branched covering $\mathfrak p:S\longrightarrow D$ with fiber $M$ over $-1\in D$, i.e. $M=\mathfrak{p}^{-1}(-1)$.
\end{definition}

Note that $\mathfrak p$ as above necessarily has $m=|M|$ sheets and $n=m+b+2g-2$ branch points by the Riemann--Hurwitz formula (see the proof of Proposition~\ref{prop:exc} below).
We typically denote the branch points by $p_1, p_2, \ldots, p_n\in D$.

\begin{definition}
Let $D$ be the disk with points $p_1, p_2,\ldots, p_n$ in its interior.
A set of \emph{matching paths} for $p_1, p_2,\ldots, p_n$ is a set $\mathbb{B}=(b_1,b_2,\ldots,b_n)$ of embedded paths in $D$ such that
\begin{enumerate}
\item $b_i$ joins $p_i$ to $-1$ and does not pass through any of the $p_j$'s in its interior,
\item $b_i\cap b_j=\{-1\}$, and
\item $b_{i+1}$ follows $b_i$ counterclockwise at $-1$, see Figure \ref{fig:matchingpaths}.
\end{enumerate}
\end{definition}

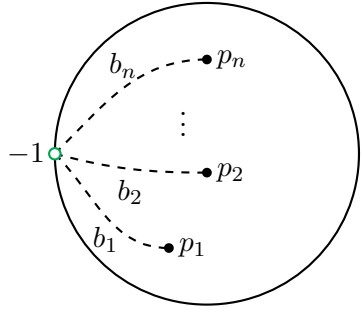
\begin{figure}[h]
\begin{center}{
\begin{tikzpicture}[scale=0.5]
\draw[thick,fill=white] (0,0) circle (4cm);		

\path (-4,0) coordinate (-1)
      (-1,-2.5) coordinate (p1)
      (0,-.5) coordinate (p2)
      (-1,1) coordinate (p3)
      (0,2.5) coordinate (p4);
\draw (-1) node[left] {$-1$};
\draw (p1) node[right] {$p_1$};
\draw (p2) node[right] {$p_2$};
\draw (p3) node[right] {$\vdots$};
\draw (p4) node[right] {$p_n$};

\draw[dashed,thick] (-1) to[out=-45,in=180] node[below]{$b_1$} (p1);
\draw[dashed,thick] (-1) to[out=-15,in=180] node[below]{$b_2$} (p2);
\draw[dashed,thick] (-1) to[out=45,in=180] node[above]{$b_n$} (p4);

\draw[Green,thick,fill=white] (-1) circle (.15cm);

\draw[thick,fill] (p1) circle (.1cm)
(p2) circle (.1cm)
(p4) circle (.1cm);	
\end{tikzpicture}}
\end{center}
\begin{center}
\caption{A collection of matching paths $b_1,\ldots,b_n$ from $-1$ to the branch points $p_1,\ldots,p_n$.}\label{fig:matchingpaths}
\end{center}
\end{figure}

For given $(S,M)$ we consider two pairs $(\mathfrak p_i,\mathbb B_i)$, $i=0,1$, consisting of a branched covering $\mathfrak p_i$ with choice of matching paths $\mathbb B_i$ \emph{equivalent} if there is a continuous family $(\mathfrak p_t,\mathbb B_t)$, $t\in [0,1]$ interpolating between them.
Note that the branch points are allowed to move along this path, and $\mathbb B_t$ must connect to the branch points of $\mathfrak p_t$ for any $t\in [0,1]$.

\subsection{From branched coverings to exceptional dissections}

Let $(S,M,\mathfrak{p},\mathbb{B})$ be a branched covering with matching paths. Then each matching path, $b_i$, lifts uniquely to a path $a_i$ in $S$ which covers it two-to-one except at the branch point.
In particular, $a_i$ is a path with (distinct) endpoints in $M$ and contains the (unique) ramification point corresponding to $p_i$.
Let $\mathbb A:=\mathfrak{p}^{-1}(\mathbb B) =(a_1,\ldots,a_n)$.

\begin{proposition}\label{prop:exc}
The set $\mathbb{A}$ is an exceptional dissection on the marked surface $(S,M)$.
\end{proposition}
\begin{proof}
Since $a_i$ is a lift of $b_i$, which has no self-intersection,  each $a_i$ is non-self-intersecting (even at the endpoints).
Since $b_i\cap b_j=-1$ for $i\neq j$, $a_i$ and $a_j$ meet at most at the endpoints. Furthermore, since each $b_{i+1}$ follows $b_i$ counterclockwise at the basepoint $-1$, if $a_j$ follows $a_i$ counterclockwise at a common endpoint, which is a lift of $-1$, then $i<j$.

Now we have to show that $a_i$ cut $S$ into polygons each of which has exactly one boundary edge. By \cite[Proposition 1.11]{APS19}, this is equivalent to the following two conditions:
\begin{enumerate}
\item every subsurface enclosed by the arcs $a_i$ contains at least one boundary segment;
\item $\mathbb{A}$ is maximal, that is, $\mathbb{A}$ has exactly $m+b+2g-2$ arcs.
\end{enumerate}
The first condition follows from the fact that there is a total order on $\mathbb{A}$. For contradiction, assume that there is a subsurface enclosed by arcs in $\mathbb{A}$, which has no boundary segment of $S$ on its boundary. Let $c_1=a_{i_1}, c_2=a_{i_2},\ldots,c_s=a_{i_s}$ be the arcs in $\mathbb{A}$ forming the boundary of such a subsurface, where we label the arcs counterclockwise, that is, $c_{i+1}$ follows $c_i$ counterclockwise at the common endpoint, with the subscript modulo $s$. Since there is no boundary segment between these arcs, they form a cycle, and we have $i_1 < i_2 < \ldots < i_s < i_1$, which is a contradiction.

The second condition follows from the classical Riemann-Hurwitz formula.
Denote by $\widetilde{S}$ the surface obtained from $S$ by shrinking each connected component of the boundary into a point. Then the branched covering $\mathfrak{p}$ can be restricted to a branched covering from $\widetilde{S}$ to $\mathbb{P}^1$, where $\mathbb{P}^1$ is obtained from $D$ by shrinking the boundary to a point $p_{n+1}$. We still denote this branched covering by $\mathfrak{p}$. Then there are exactly $b$ points in the fiber $\mathfrak{p}^{-1}(p_{n+1})$, which are the points coming from shrinking the boundary components of $S$.

The Riemann-Hurwitz formula states that
\begin{equation}\label{eq:RieHur}
\chi(S)=m\chi(\mathbb{P}^1)-\sum (\deg(\mathfrak{p};x)-1)
\end{equation}
where $\chi$ denotes the Euler characteristic of the surfaces and the summation extends over all singular points $x$ of $\mathfrak{p}$.
Since $\mathfrak{p}$ is $m$-sheeted, the sum of $\deg(\mathfrak{p};x)$ over the fiber of each branched point in $\mathbb{P}^1$ equals $m$.
In particular, we have
\begin{equation}
\sum (\deg(\mathfrak{p};x)-1)=m-b
\end{equation}
where the summation extends over all the $b$ singular points in $\mathfrak{p}^{-1}(p_{n+1})$. On the other hand, for each singular point $x$ in $\mathfrak{p}^{-1}(p_{i}), 1\leqslant i\leqslant n,$ $\deg(\mathfrak{p};x)=2$, thus
\begin{equation}
\sum (\deg(\mathfrak{p};x)-1)=n
\end{equation}
where the summation extends over all the $n$ singular points in $\mathfrak{p}^{-1}(p_{i})$, $1\leqslant i\leqslant n$.

At last, note that $\chi(\widetilde{S})=\chi(S)=2-2g$, and $\chi(\mathbb{P}^1)=2$, thus by the equalities above, we have
$n=m+b+2g-2$.
So $\mathbb{A}$ is maximal, and we are done.
\end{proof}

The next proposition states that the above construction is compatible with our notion of equivalence.

\begin{proposition}
Let $(S,M)$ be a marked surface and let $(\mathfrak p_i,\mathbb B_i)$, $i=0,1$, be two equivalent branched coverings $\mathfrak p_i:(S,M)\to D$ together with choice of matching paths $\mathbb B_i$.
Then the associated exceptional dissections $\mathbb A_i:=\mathfrak p_i^{-1}(\mathbb B_i)$, $i=0,1$, are equivalent (i.e. isotopic).
\end{proposition}
\begin{proof}
By definition, there is a continuous path $(\mathfrak p_t,\mathbb B_t)$, $t\in[0,1]$, interpolating between the two.
This gives rise to a path of exceptional dissections $\mathbb A_t:=\mathfrak p_t^{-1}(\mathbb B_t)$, proving the claim.
\end{proof}

By the above two propositions, we can define a map from $\mathfrak{S}_2$ to $\mathfrak{S}_1$ as follows:

\begin{proposition-definition}\label{prop-def:phi1}
There is a well-defined map $\Phi: \mathfrak{S}_2\rightarrow \mathfrak{S}_1$, which maps a branched covering $\mathfrak p$ with matching paths $\mathbb{B}$ to the exceptional dissection $\mathfrak{p}^{-1}(\mathbb{B})$.
\end{proposition-definition}

\subsection{From exceptional dissections to branched coverings}

The following proposition shows that $\Phi$ is surjective.

\begin{proposition}\label{prop:branchedconstruction}
Let $(S,M)$ be a marked surface and $\mathbb A=(a_1,\ldots,a_n)$ an exceptional dissection on it.
Then there exists a branched covering $\mathfrak p:(S,M)\to D$ together with matching paths $\mathbb B$ such that $\mathbb A= \mathfrak p^{-1}(\mathbb B)$.
\end{proposition}

\begin{proof}
Choose branch points $p_1,\ldots,p_n$ and matching paths $\mathbb B=\{b_1,\ldots,b_n\}$ in $D$.
Assume this is done in such a way that the straight horizontal path $c_i$ from $p_i$ to $\partial D$ does not intersect any of the $b_i$'s away from $p_i$ (see Figure~\ref{fig:diskwithcuts}).

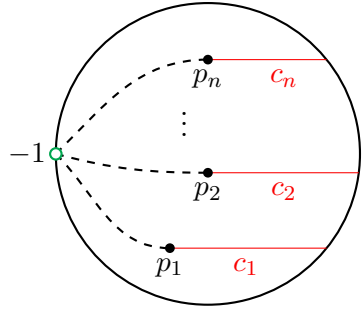
\begin{figure}[h]
\begin{center}{
\begin{tikzpicture}[scale=0.5]
\draw[thick,fill=white] (0,0) circle (4cm);		

\path (-4,0) coordinate (-1)
      (-1,-2.5) coordinate (p1)
      (0,-.5) coordinate (p2)
      (-1,1) coordinate (p3)
      (0,2.5) coordinate (p4);
\draw (-1) node[left] {$-1$};
\draw (p1) node[below] {$p_1$};
\draw (p2) node[below] {$p_2$};
\draw (p3) node[right] {$\vdots$};
\draw (p4) node[below] {$p_n$};

\draw[dashed,thick] (-1) to[out=-45,in=180] (p1);
\draw[dashed,thick] (-1) to[out=-15,in=180] (p2);
\draw[dashed,thick] (-1) to[out=45,in=180] (p4);

\draw[Green,thick,fill=white] (-1) circle (.15cm);	

\draw[red] (p1) to (3.12,-2.5);
\draw (1,-2.5) node[below] {$\color{red} c_1$};
\draw[red] (p2) to (3.97,-.5);
\draw (2,-.5) node[below] {$\color{red} c_2$};
\draw[red] (p4) to (3.12,2.5);
\draw (2,2.5) node[below] {$\color{red} c_n$};

\draw[thick,fill] (p1) circle (.1cm)
(p2) circle (.1cm)
(p4) circle (.1cm);	
\end{tikzpicture}}
\end{center}
\begin{center}
\caption{Copy of the disk $D$ cut along paths $c_1,c_2,\ldots,c_n$ avoiding the matching paths (dashed).}\label{fig:diskwithcuts}
\end{center}
\end{figure}

We construct a surface $S'$ from $M\times D$ by cutting and gluing:
\begin{enumerate}
\item Cut $\{x\}\times D$ along all paths $c_i$ where $i$ is such that $a_i$ ends at $x$.
\item Glue $\{x\}\times D$ to $\{y\}\times D$ along the cuts, where $\{x,y\}$ are the endpoints of $a_i$.
\end{enumerate}
By construction, $S'$ comes with a map $\mathfrak p':S'\to D$ which is a simple branched covering with branch points $p_1,\ldots,p_n$.
Over $p_i$, the two sheets corresponding to the endpoints of $a_i$ come together.
Furthermore, if $M':=\mathfrak p'^{-1}(-1)\subset\partial S'$ (which is canonically identified with $M$), then $(S',M')$ is a marked surface with exceptional dissection $\mathbb A':=\mathfrak p'^{-1}(\mathbb B)=(a_1',\ldots,a_n')$.

We construct a diffeomorphism of marked surfaces $\varphi:(S,M)\to (S',M')$ as follows.
First, $\varphi|M$ is defined by the canonical bijection $M\cong M'$.
Let $\Gamma$ (resp. $\Gamma'$) be the graph formed by the union of the $a_i$ (resp $a_i'$).
Then extend $\varphi$ to $\Gamma$ is by sending $a_i$ to $a_i'$ (uniquely, up to reparametrization along each arc).
We can do this, since the endpoints of $a_i$ and $a_i'$ are by construction the same, once we identify $M$ and $M'$.
Also note that $\varphi|\Gamma$ is an isomorphism of ribbon graphs (i.e. extends to a map from a neighbourhood of $\Gamma$ to a neighbourhood of $\Gamma'$), since the ribbon structure corresponds to the order of the arcs in both cases.
Finally, extend $\varphi$ to all of $S$ by mapping each polygon in $S$ cut out by $\mathbb A$ to the corresponding polygon in $S'$ cut out by $A'$.

Finally, set $\mathfrak p:=\mathfrak p'\circ \varphi$, then $\mathfrak p:(S,M)\to D$ is a branched covering and $\mathfrak p^{-1}(\mathbb B)=\varphi^{-1}(\mathbb A')=\mathbb A$, which completes the proof.
\end{proof}

The following proposition shows that $\Phi$ is injective, thus, since we already know that $\Phi$ is surjective, a bijection.

\begin{proposition}
Let $(S,M)$ be a marked surface and suppose that two pairs $(\mathfrak p_i,\mathbb B_i)$, $i=0,1$, consisting of a branched covering $\mathfrak p_i$ and a choice of matching paths $\mathbb B_i$ give rise to equivalent exceptional dissections $\mathbb A_i=\mathfrak p_i^{-1}(\mathbb B_i)$.
Then $(\mathfrak p_0,\mathbb B_0)$ is equivalent to $(\mathfrak p_1,\mathbb B_1)$.
\end{proposition}

\begin{proof}
Since $\mathbb A_0$ is equivalent to $\mathbb A_1$, there is a diffeomorphism $\varphi$, isotopic to the identity, with $\varphi(\mathbb A_0)=\mathbb A_1$.
Then replacing $\mathfrak p_1$ by $\mathfrak p_1\circ\varphi$, which does not change the equivalence class of $(\mathfrak p_1,\mathbb B_1)$, we can assume that $\mathbb A_0=\mathbb A_1=(a_1,\ldots,a_n)$.

Next, we can find a diffeomorphism $\psi:D\to D$, isotopic to the identity, with $\psi(\mathbb B_0)=\mathbb B_1$.
Thus, replacing $\mathfrak p_0$ by $\psi\circ\mathfrak p_0$ and $\mathbb B_0$ by $\mathbb B_1$, which does not change the equivalence class, we can assume that $\mathbb B_0=\mathbb B_1=(b_1,\ldots,b_n)$.

Note that both $\mathfrak p_0$ and $\mathfrak p_1$ map $a_i$ to $b_i$, but possibly with different parametrization.
Composing $\mathfrak p_0$ with another diffeomorphism $S\to S$ isotopic to the identity, we can assume that $\mathfrak p_0=\mathfrak p_1$ along each $a_i$.
The $\mathfrak p_i$'s can still differ in the interior of the polygons cut out by $\mathbb A$, but composing $\mathfrak p_0$ with another diffeomorphism $S\to S$ isotopic to the identity and fixing $\mathbb A$, we can achieve $\mathfrak p_0=\mathfrak p_1$ everywhere.
\end{proof}

This completes the first half of the proof of Theorem~\ref{main theorem bijections}, showing that there is a canonical bijection between $\mathfrak S_1$ and $\mathfrak S_2$.

\subsection{Hurwitz systems}

The following notion was considered by Hurwitz in~\cite{Hur1891}.

\begin{definition}
Fix a finite set $M$.
A \emph{Hurwitz system} is an ordered set of $n$ transpositions $\tau=(\tau_1, \tau_2, \ldots, \tau_n)$ of $M$, which generate the whole symmetric group of $M$.
\end{definition}

A marked surface $(S,M)$ with exceptional dissection $\mathbb A=(a_1,\ldots a_n)$ gives rise to a Hurwitz system $\tau=(\tau_1,\ldots,\tau_n)$ of $M$ where $\tau_i$ is the transposition of the two endpoints of $a_i$.
The fact that the $\tau_i$'s generate the full symmetric group of $M$ follows from the connectedness of $S$:

\begin{lemma}
The ordered set $\tau=(\tau_1,\tau_2,\ldots,\tau_n)$ is a Hurwitz system.
\end{lemma}
\begin{proof}
We have to show that $\tau_i$'s generate the whole symmetric group of $M$.
For this, we show that any transposition of $M$ can be generated by $\tau_i$'s.
Let $\tau_i=(x_i,y_i)$.
Note that if $y_i=x_j$ and $x_i\neq y_j$, then $\tau_i\tau_j\tau_i=(x_i,y_j)$, and in this case, the smoothing of $a_i$ and $a_j$ at $y_i=x_j$ is an arc connecting $x_i$ and $y_j$.
On the other hand, since $S$ is connected and $\mathbb{A}$ is in partciular an admissible dissection, any arc (with any endpoints) can be obtained by smoothing a sequence of arcs in $\mathbb{A}$.
Thus any transposition of $M$ can be generated by $\tau_i$.
\end{proof}

By combining the above constructions, it follows that a branched covering $\mathfrak p:(S,M)\to D$ with matching paths $\mathbb B$ gives rise to a Hurwitz system where $\tau_i$ interchanges the two sheets which come together over the branch point $p_i$, and where we identify  sheets along $b_i$.

\begin{definition}
{A \emph{Hurwitz system of type $(S,M)$} is a Hurwitz system $\tau=(\tau_1,\ldots,\tau_n)$ arising from some exceptional dissection $\mathbb A$ on $(S,M)$ as above.}
\end{definition}

{By definition, we get a surjective map $\mathfrak S_1\to\mathcal H$, where $\mathcal H$ is the set of Hurwitz systems of type $(S,M)$ as before.
Moreover, this induces a map $\mathfrak S_1/\mathrm{Mod}(S)\to\mathcal H$, since elements of $\mathrm{Mod}(S)$ are required to fix $\partial S$.
To complete the proof of Theorem~\ref{main theorem bijections} we need to show that this map is injective. 
Suppose $\mathbb A$ and $\mathbb A'$ are two exceptional dissections on $(S,M)$ giving rise to the same Hurwitz system. 
Then as in the proof of Proposition~\ref{prop:branchedconstruction} on constructs a diffeomorphism $\varphi:(S,M)\to(S,M)$ fixing $\partial S$ and sending $\mathbb A$ to $\mathbb A'$.
Thus, both exceptional dissections define the same element in $\mathfrak S_1/\mathrm{Mod}(S)$.}

\section{Braid group actions}

Denote by $\mathfrak{B}_n$ the Artin braid group generated by $\sigma_1,\ldots,\sigma_{n-1}$
with relations
$\sigma_i\sigma_j=\sigma_j\sigma_i, |i-j|>1$ and
$\sigma_i\sigma_{i+1}\sigma_i=\sigma_{i+1}\sigma_i\sigma_{i+1}$.

In this section, we will introduce actions of  $\mathfrak{B}_n$ on the sets $\mathfrak{S}_1$, $\mathfrak{S}_2$ and {$\mathcal{H}$}, and show the following

\begin{theorem}\label{theorem:groupaction}
The braid group actions on $\mathfrak{S}_1$, $\mathfrak{S}_2$, and {$\mathcal{H}$} are compatible with the bijections established in Theorem \ref{main theorem bijections}.
\end{theorem}

This follows for $\mathfrak S_1$ and $\mathfrak S_2$ from Proposition~\ref{prop:braids1s2} below, and for $\mathfrak S_1$ and  {$\mathcal{H}$ from the discussion in Section~\ref{subsec:brhurwitz}}.

\subsection{Braid group action on exceptional dissections}

Let $(a_1,a_2)$ be an ordered exceptional pair in $(S,M)$, that is, $a_1$ and $a_2$ either do not intersect or they intersect at one or two common endpoints, and in this case $a_2$ follows $a_1$ counterclockwise at each common endpoint, as shown in Figure~\ref{figure:mutation of exceptional dissection}. The following definition is introduced in \cite{CS22}.

\begin{definition}\label{definition:mutations}
(1) Let $(a_1,a_2)$ be an ordered exceptional pair in $(S,M)$. We define \emph{the left (resp. right) mutation of $a_2$ at $a_1$} denoted by $L_{a_1}a_2$ (resp. $R_{a_2}a_1$) as follows:
\begin{itemize}
\item If $a_1$ and $a_2$ do not intersect, then we define $L_{a_1}a_2$ as $a_2$ and define $R_{a_2}a_1$ as $a_1$.
\item  If $a_1$ and $a_2$ share one endpoint $q_1$ then $L_{a_1}a_2$ is the smoothing of the crossing of $a_1$ and $a_2$  at $q_1$.
\item If $a_1$ and $a_2$ share both endpoints, then   $L_{a_1}a_2$ is obtained by first smoothing the crossing of  $a_1$ and $a_2$ at one of the endpoints and then smoothing the crossing of the resulting arc with $a_1$ at the other endpoint of $a_1$, see  Figure \ref{figure:mutation of exceptional dissection}.
\end{itemize}
Dually, we define $R_{a_2}a_1$ when $a_1$ and $a_2$ intersect in one or both endpoints, see  Figure \ref{figure:mutation of exceptional dissection}.

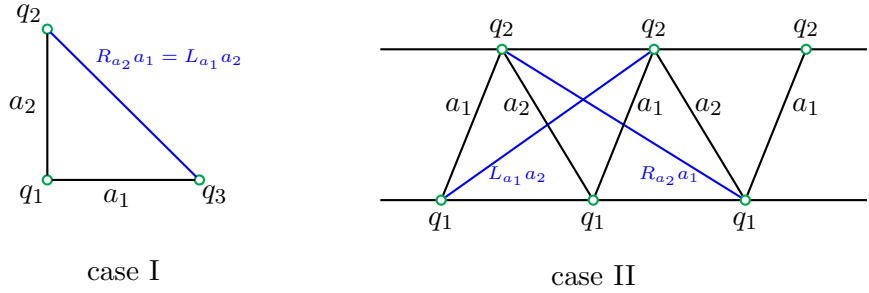
\begin{figure}[H]
\begin{center}
{\begin{tikzpicture}[scale=0.4]
\draw[thick,Green] (0,0) circle [radius=0.15];
\draw[thick,Green] (0,5) circle [radius=0.15];
\draw[thick,Green] (5,0) circle [radius=0.15];

\draw[thick,-] (0,0.15) -- (0,4.85);

\draw[thick,-] (0.15,0) -- (4.85,0);

\draw[thick,-,blue] (0.1,4.9) -- (4.9,0.1);

\draw (2.3,-.6) node {$a_1$};

\draw (-.8,2.5) node {$a_{2}$};
\draw [blue](4,4) node {\tiny $R_{a_2}a_1=L_{a_1}a_2$};
\draw (-.6,5.5) node {$q_2$};
\draw (-.5,-.5) node {$q_1$};
\draw (5.5,-.5) node {$q_3$};
%

\draw (2.5,-3) node {case I};
\draw (2.5,-3.5) node {};

\end{tikzpicture}}
\qquad\qquad
{\begin{tikzpicture}[scale=0.4]
\draw[thick,-,blue]  (-3,2.5)to(5,-2.5);
\draw[thick,-,blue]  (2,2.5)to(-5,-2.5);

\draw[thick,-]  (-7,-2.5)to(9,-2.5);
\draw[thick,-]  (-7,2.5)to(9,2.5);

\draw (-4.4,0.6) node {$a_1$};
\draw (1.9,.6) node {$a_1$};
\draw (7,0.6) node {$a_1$};

\draw (-2.5,.6) node {$a_2$};
\draw (3.8,0.6) node {$a_2$};

\draw (-5,-3.2) node {$q_1$};
\draw (0,-3.2) node {$q_1$};
\draw (5,-3.2) node {$q_1$};

\draw[thick,-]  (-5,-2.5)to(-3,2.5);
\draw[thick,-]  (0,-2.5)to(-3,2.5);
\draw[thick,-]  (0,-2.5)to(2,2.5);
\draw[thick,-]  (5,-2.5)to(2,2.5);
\draw[thick,-]  (5,-2.5)to(7,2.5);

\draw[thick,Green,fill=white] (0,-2.5) circle [radius=0.15];
\draw[thick,Green,fill=white] (5,-2.5) circle [radius=0.15];
\draw[thick,Green,fill=white] (-5,-2.5) circle [radius=0.15];

\draw[thick,Green,fill=white] (7,2.5) circle [radius=0.15];
\draw[thick,Green,fill=white] (2,2.5) circle [radius=0.15];
\draw[thick,Green,fill=white] (-3,2.5) circle [radius=0.15];

\draw (-3,3.2) node {$q_2$};
\draw (2,3.2) node {$q_2$};
\draw (7,3.2) node {$q_2$};

\draw[blue] (-2.5,-1.7) node {\tiny $L_{a_1}a_2$};
\draw[blue] (2.5,-1.7) node {\tiny $R_{a_2}a_1$};

\draw (0,-5) node {case II};
\draw (2.5,-3.5) node {};

\end{tikzpicture}}
\end{center}

\begin{center}
\caption{Possible intersections of  $a_{1}$ and $a_{2}$ in an ordered exceptional pair $(a_1,a_2)$, where case II is depicted in the universal covering. This figure also illustrates the mutations of the ordered exceptional pair $(a_1,a_2)$ (in blue).}\label{figure:mutation of exceptional dissection}
\end{center}
\end{figure}
(2) Let $\mathbb{A}=(a_1,\ldots,a_n)$ be an exceptional dissection of arcs in $(S,M)$, for any $1 \leq i \leq n-1$, we define
\vspace{-.2cm}
$$\begin{array} {l}
\sigma_i \mathbb{A}=(a_1,\ldots,a_{i-1},a_{i+1}
,R_{a_{i+1}}a_i,a_{i+2},\ldots,a_n),
\\
\\
\sigma_i^{-1} \mathbb{A}=(a_1,\ldots,a_{i-1},L_{a_{i}}a_{i+1},a_{i}
,a_{i+2},\ldots,a_n).\\
\end{array}$$
\end{definition}

\medskip

\begin{theorem}\cite{CS22}\label{theorem:mutation}
Let $\mathbb{A}=(a_1,\ldots,a_n)$ be an exceptional dissection on $(S,M)$. Then  $\sigma_i \mathbb{A}$ and $\sigma^{-1}_i \mathbb{A}$ are exceptional dissections on $(S,M)$, for all $1 \leq i \leq n-1$.
Furthermore, this induces an action of the braid group $\mathfrak B_n$ on the set of exceptional dissections on $(S,M)$.
\end{theorem}

\subsection{Braid group action on branched coverings with matching paths}

Let $(S,M)$ be a marked surface, $\mathfrak p:(S,M)\to D$ a branched covering with branch points $p_1,\ldots,p_n$, and $\mathbb B$ a choice of matching paths.
The mapping class group $\mathrm{Mod}(D\setminus \{p_1,\ldots,p_n\})$ of diffeomorphisms fixing the boundary and mapping punctures to punctures, up to isotopy, is isomorphic to the braid group $\mathfrak B_n$.
A choice of isomorphism is fixed by $\mathbb B$ as follows.
First, there is total order on the set of matching paths corresponding to the order in which they meet $-1\in D$.
Thus there is a total order on the set of branch points which we can assume to be $p_1<p_2<\ldots<p_n$.
Choose an embedded path $e_i$ which connects $p_i$ to $p_{i+1}$, $i=1,\ldots,n-1$, and sends the generator $\sigma_i\in\mathfrak B_n$ to the \emph{half-twist} along $e_i$, i.e. a diffeomorphism of $S$ which interchanges $p_i$ and $p_{i+1}$ and is the identity outside a neighbourhood of $e_i$,(see Figure~\ref{fig:braidgroupaction}).

Using the above identification, we define the action of $\sigma\in\mathfrak{B}_n$ on $(\mathfrak p,\mathbb B)$ by sending $\mathbb B$ to its image $f(\mathbb B)$ under the element $f$ of $\mathrm{Mod}(D\setminus \{p_1,\ldots,p_n\})$ corresponding to $\sigma$. 
The result is well-defined up to equivalence, i.e. well defined as an element of $\mathfrak S_2$.
Equivalently, we can instead send $\mathfrak p$ to $f\circ \mathfrak p$.
In the second description it is easy to see that we get an action of $\mathfrak{B}_n$, since the identification of the braid group with $\mathrm{Mod}(D\setminus \{p_1,\ldots,p_n\})$ stays the same.

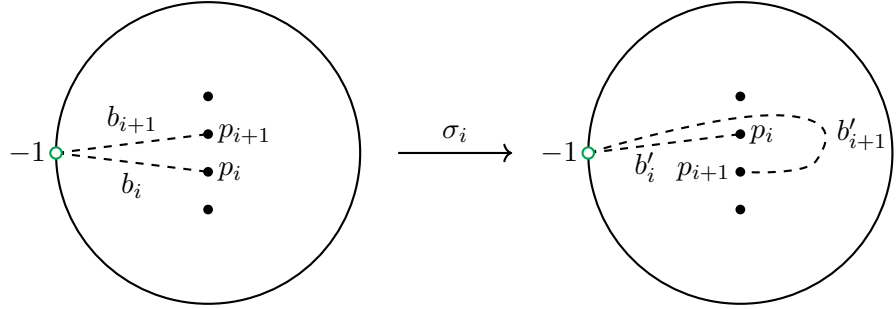
\begin{figure}[h]
\begin{center}{
\begin{tikzpicture}[scale=0.5]
\draw[thick,fill=white] (0,0) circle (4cm);		

\path (-4,0) coordinate (-1)
      (0,1.5) coordinate (bi-1)
      (0,-.5) coordinate (bi)
      (0,.5) coordinate (bi+1)
      (0,-1.5) coordinate (bi+2);
\draw (-1) node[left] {$-1$};
\draw (bi) node[right] {$p_i$};
\draw (bi+1) node[right] {$p_{i+1}$};

\draw[dashed,thick] (-1) to node[below]{$b_i$} (bi);
\draw[dashed,thick] (-1) to node[above]{$b_{i+1}$} (bi+1);

\draw[Green,thick,fill=white] (-1) circle (.15cm);

\draw[thick,fill] (bi) circle (.1cm)
(bi-1) circle (.1cm)
(bi+1) circle (.1cm)
(bi+2) circle (.1cm);	
\draw[thick,fill=white] (14,0) circle (4cm);		

\path (10,0) coordinate (-1)
      (14,1.5) coordinate (bi-1)
      (14,-.5) coordinate (bi)
      (14,.5) coordinate (bi+1)
      (14,-1.5) coordinate (bi+2);
\draw (-1) node[left] {$-1$};
\draw (bi) node[left] {$p_{i+1}$};
\draw (bi+1) node[right] {$p_i$};
\draw (17.2,.5) node {$b'_{i+1}$};

\draw[dashed,thick]plot [smooth,tension=1] coordinates {(-1) (15,1) (16,-.2) (bi)};

\draw[dashed,thick] (-1) to (bi+1);
\draw (11.5,0.3) node[below]{$b'_{i}$};
\draw[Green,thick,fill=white] (-1) circle (.15cm);

\draw[thick,fill] (bi) circle (.1cm)
(bi-1) circle (.1cm)
(bi+1) circle (.1cm)
(bi+2) circle (.1cm);	

\draw[->,thick] (5,0) to node[above]{$\sigma_{i}$} (8,0);
\end{tikzpicture}}
\end{center}
\begin{center}
\caption{The action of the generator $\sigma_i$ in the braid group $\mathfrak{B}_n$ on matching paths.}\label{fig:braidgroupaction}
\end{center}
\end{figure}

\begin{proposition}\label{prop:braids1s2}
The bijection $\Phi:\mathfrak S_2\to\mathfrak S_1$ intertwines the two braid group actions defined above.
\end{proposition}

\begin{proof}
Let $\sigma_i\in\mathfrak B$ be a standard generator, $\mathfrak p:(S,M)\to D$ a branched covering, and $\mathbb B=(b_1,\ldots,b_n)$ a choice of matching paths.
Then $\sigma_i\mathbb B=(b_1,\ldots,b_{i-1},b_i',b_{i+1}',b_{i+2},\ldots,b_n)$ where $b_i'$ and $b_{i+1}'$ are depicted in Figure~\ref{fig:braidgroupaction}.
Let $a_i,a_{i+1},a_i',a_{i+1}'$ be the arcs in $S$ corresponding to $b_i,b_{i+1},b_i',b_{i+1}'$, respectively.
Since $b_i'=b_{i+1}$, it is clear that $a_i'=a_{i+1}$.
It remains to show that $a_{i+1}'$ is homotopic to $R_{a_{i+1}}a_i$.
There are three cases, as in the definition of $R_{a_{i+1}}a_i$.

In the first case, $a_i$ and $a_{i+1}$ do not intersect, so we can slide $b_{i+1}'$ back to $b_i$ and this lifts to a homotopy of paths from $a_{i+1}'$ to $a_i$, since the two sheets that $a_i$ lies on are not ramified at $p_{i+1}$.
Thus $a_{i+1}'$ is homotopic to $a_i$, which is $R_{a_{i+1}}a_i$ by definition.

In the second case, using a similar idea but moving on only one of the two sheets, we can find a homotopy between $a_{i+1}'$ and a path $\tilde{a}$ which projects to a path that follows $b_i$, a path from $p_i$ to $p_{i+1}$, and $b_{i+1}$.
A further homotopy of the middle part of $\tilde{a}$ deforms this to a smoothing of the concatenation of $a_i$ and $a_{i+1}$, which is $R_{a_{i+1}}a_i$ by definition.

In the third case only two sheets are involved. Then $b_{i+1}'$ is homotopic to a path which follows $b_{i+1}$ from $-1$ to $p_{i+1}$, follows $b_{i+1}$ back to $-1$, and follows $b_i$ to $p_i$.
Correspondingly, $a_{i+1}'$ is homotopic to a path which follows $a_i$, then $a_{i+1}$, and then $a_i$ again, and is thus homotopic to $R_{a_{i+1}}a_i$ by definition.
\end{proof}

\subsection{Braid group action on Hurwitz systems}\label{subsec:brhurwitz}

{The action of the generator $\sigma_i$ of $\mathfrak{B}_n$ on Hurwitz systems is defined by
$$\sigma_i: (\tau_1,\ldots,\tau_n)\mapsto
(\tau_1,\ldots,\tau_{i-1},\tau_{i+1},\tau_{i+1}\tau_i\tau_{i+1},\tau_{i+2},\ldots,\tau_n).$$
By inspection of Figure~\ref{figure:mutation of exceptional dissection} (see also~\cite{Hur1891}) one sees that this is compatible with the $\mathfrak B_n$-action on $\mathfrak S_1$ in the sense that the map $\mathfrak S_1\to\mathcal H$ intertwines the two actions.}

In~\cite{K88}, Kluitmann proves the following transitivity of the braid group action on Hurwitz systems.

\begin{theorem}[\cite{K88}]
\label{thm:kluitmann}
The braid group $\mathfrak B_n$ acts transitively on the set of Hurwitz systems {of type $(S,M)$}.
\end{theorem}

Actually, Kluitmann classifies the orbits of $\mathfrak B_n$ and shows that they correspond exactly to the cycle type of the permutation $\tau_1\tau_2\cdots \tau_n$.
But cycle types correspond precisely to marked surfaces up to isomorphism, where a $k$-cycle corresponds to a component of $\partial S$ with $k$ points on the boundary and the genus of $S$ is determined by the relation $n=m+b+2g-2$.

In contrast, we will see in Section \ref{section:counterexamples} that the action of $\mathfrak B_n$ on {$\mathfrak S_1\cong\mathfrak S_2$} is not always transitive.

\section{Applications to Fukaya categories of surfaces}

In this section we relate the geometric considerations of the previous two sections to full exceptional sequences in Fukaya categories of surfaces.
In this section and the following we fix an algebraically closed coefficient field $\mathbf{k}$.

\subsection{Fukaya categories of graded surfaces}  \label{sec:fukaya}

In contrast to the higher dimensional case, Fukaya categories of surfaces have much simpler definitions.
Here, we recall the approach from~\cite{HKK17}.
The input data for the construction is:
\begin{itemize}
\item
a compact oriented surface $S$, possibly with boundary,
\item
\textit{marked points:} a non-empty finite subset $M\subset S$, such that each component of $\partial S$ contains at least one point of $M$, and
\item
\textit{grading structure:} a line field $\nu\in\Gamma(S,\mathbb P(TS))$, i.e. a choice of tangent line $\nu_p\subset T_pS$, varying smoothly with $p\in S$.
\end{itemize}

From the above data a triangulated $A_\infty$-category, the \textbf{(partially wrapped) Fukaya category}, $\mathcal F_\infty(S,M,\nu)$ (whose triangulated homotopy category is denoted here by $\mathcal F(S,M,\nu):=H^0(\mathcal F_\infty(S,M,\nu))$) is constructed in~\cite{HKK17}.
We review the parts of the construction which are relevant for the present work.

First, an auxiliary choice of \textit{arc system} on $S$ is made.
This is a collection, $\mathbb A$, of immersed arcs in $S$ whose endpoints belong to $M$, which intersect themselves and each other transversely and only in endpoints, and which cut $S$ into polygons which contain no marked points in their interior and whose edges are arcs belonging to $\mathbb A$.
In particular, each interval in $\partial S$ between adjacent marked points must belong to $\mathbb A$.

Each $a\in \mathbb A$ will determine an object in $\mathcal F_\infty(S,M,\nu)$ up to shift.
To fix the shift, a choice of \textit{grading} on the arcs is needed, which we include henceforth as part of the structure of an arc system.
A \textit{graded curve} in $S$ is a triple $(I,c,\tilde{c})$ where $c:I\to S$ is an immersed curve and $\tilde{c}$ is a homotopy class of paths in $\Gamma(I,c^*\mathbb P(TS))$ from $c^*\nu$ to $\dot{c}$, where $\dot c$ is the section given by the tangent lines to $c$.
Concretely: If $t\in I$ then $T_{c(t)}S$ contains two lines: $\nu_{c(t)}$ and the tangent line to $c$.
To give $\tilde{c}$ is to choose a way of interpolating between the two, consistently as one varies $t$.
If $I=[0,1]$, then the possible choices of $\tilde{c}$ form a $\mathbb Z$-torsor, where $1\in\mathbb Z$ acts by adding a \emph{counterclockwise} rotation of angle $\pi$ to $\tilde c$.
For convenience, we will sometimes abuse the notation and denote the graded curve by $c$.

The next step is to consider the real blow-up, $\widehat{S}$, of $S$ in $M$.
This is the surface with corners where each $p\in M\setminus \partial S$ has been replaced by a boundary circle and each $p\in M\cap\partial S$ has been replaced by an interval whose endpoints are corners of $\widehat S$.
There is a canonical map $\widehat{S}\to S$ which collapses these new parts of the boundary, the \textit{marked boundary}, to $M$.
(In~\cite{HKK17} the surface with corners $\widehat S$ was, equivalently, taken as a starting  point instead of $S$ itself.)
The grading $\nu$ lifts to $\widehat S$, possibly after small perturbation near $M$, and the arcs in $\mathbb A$ lift to non-intersecting embedded intervals on $\widehat S$ which end in the marked boundary.
The advantage of passing to $\widehat S$ is that we can speak about \textit{boundary paths}, which are immersed paths, up to reparametrization, that start and end at arcs in $\mathbb A$ and follow the marked boundary in the direction where the interior of the surface stays to the right of the path (see Figure~\ref{fig:g1b2}).

The choice of grading of the arcs in $\mathbb A$ allows us to assign an integer degree, $|\za|$, to any boundary path $\za$.
First, suppose $(I_i,c_i,\tilde{c}_i)$, $i\in\{1,2\}$, are any two graded curves in $S$ intersecting transversely at $p=c_1(t_1)=c_2(t_2)$.
Then the \textit{intersection index} of $c_1$ and $c_2$ at $p$ is
\[
i_p(c_1,c_2):=\tilde{c}_1(t_1)\cdot \kappa\cdot\tilde{c}_2(t_2)^{-1}\in\pi_1(\mathbb P(T_pS))=\mathbb Z
\]
where $\kappa$ is the minimal counterclockwise rotation to get from $\dot{c}_1(t_1)$ to $\dot{c}_2(t_2)$. In particular, we have
\begin{equation}\label{eq:index}
i_p(c_1,c_2)+i_p(c_2,c_1)=1.
\end{equation}
To define $|\za|$ for a boundary path $\za$ starting at $a\in \mathbb A$ and ending at $b\in\mathbb A$, choose a grading on $\za$ and let $|\za|:=i_p(a,\za)-i_p(b,\za)$.
This is additive under concatenation of boundary paths.

The construction in~\cite{HKK17} proceeds by defining an $A_\infty$-category $\mathcal F_{\mathbb A}$ whose objects are elements of $\mathbb A$, whose morphisms from $a$ to $b$ have a basis given by boundary paths from $a$ to $b$, and, if $a=b$, the identity morphism, and structure constants which have two sources: concatenation of boundary paths and counting of immersed polygons whose boundary maps to arcs and boundary paths.
The category $\mathcal F_\infty(S,M,\nu)$ is then defined as the category $\mathrm{Tw}(\mathcal F_{\mathbb A})$ of (one-sided) twisted complexes over $\mathcal F_{\mathbb A}$  --- a concrete model for the closure under shifts and cones of $\mathcal F_{\mathbb A}$.
This category is shown to be independent of the choice of $\mathbb A$ up to canonical equivalence.

In the case $\partial S\neq\emptyset$, which includes the cases of interest in this work, the partially wrapped Fukaya category has a simpler description in terms of \textit{graded gentle algebras}.
The point is that under the condition $\partial S\neq\emptyset$ we can find a \textit{full formal arc system}, $\mathbb A$, which is similar to an arc system, but instead of requiring that all edges of any polygon cut out by $\mathbb A$ belong to $\mathbb A$, we require instead that all but exactly one of the edges of any polygon belong to $\mathbb A$.
The edges which do not belong to $\mathbb A$ must be boundary arcs, so we can always obtain an arc system $\mathbb A'$ from a full formal arc system $\mathbb A$ by adding all boundary arcs.
There is a full subcategory $\mathcal F_{\mathbb A}\subset \mathcal F_{\mathbb A'}$ whose objects are precisely those arcs belonging to $\mathbb A$, and which has the following properties~\cite[Section 3.4]{HKK17}:
\begin{enumerate}
\item
All $A_\infty$ structure maps $\mathfrak m_k$ vanish for $k\neq 2$, i.e. $\mathcal F_{\mathbb A}$ is essentially a graded linear category.
\item
$\mathcal F_{\mathbb A}$ generates $\mathcal F_{\mathbb A'}$ under shifts and cones, so
\[
\mathrm{Tw}(\mathcal F_{\mathbb A})=\mathrm{Tw}(\mathcal F_{\mathbb A'})=\mathcal F_\infty(S,M,\nu)
\]
\end{enumerate}
In particular, the direct sum of objects corresponding to arcs of a full formal arc system is a formal generator of $\mathcal F_\infty(S,M,\nu)$.
The product of basis morphisms in $\mathcal F_{\mathbb A}$ is concatenation of paths (and is zero if the paths cannot be concatenated).

We call the category $\mathcal F(S,M,\nu):=H^0(\mathcal F_\infty(S,M,\nu))$ the \emph{topological Fukaya category} associated to the graded surface.
Then $\mathcal F(S,M,\nu)$ is a triangulated category which is triangle equivalent to the derived category of graded gentle algebras arising from the full formal arc systems on the surface.
{There are nice descriptions of the indecomposable objects and morphisms between them in $\mathcal F(S,M,\nu)$, which we summarise in Theorem~\ref{lemma:hom-int}, where the first part is due to \cite{HKK17} and the second part (on morphisms) is due to \cite[Lemma 9]{T22}, see also \cite{IQZ20}.}

\begin{theorem}\cite{HKK17,T22,IQZ20}\label{lemma:hom-int}
There is a bijection $X$ between
the set of isotopy classes of graded curves $\{(I,c,\tilde{c})\}$ on $(S,M,\nu)$ with local system
and the set of isomorphism classes of indecomposable objects $\{X_{c}\}$ in $\mathcal F(S,M,\nu)$.
Furthermore, let $(I_1,c_1,\tilde{c}_1)$, $(I_2,c_2,\tilde{c}_2)$ be two graded curves which are not closed curves (and hence no local system is needed). 
Then each index $\rho$ intersection between them induces a (non-trivial) morphism in
$\Hom^\rho(X_{c_1}, X_{c_2})$.
Moreover, when $c_1\neq c_2$, these morphisms form a basis for the $\Hom^\bullet$ space so that we have
\begin{gather}\label{eq:int}
    \dim\Hom^\bullet(X_{c_1}, X_{c_2})=\Int({c_1},{c_2}),
\end{gather}
where
\[\Int(c_1,c_2)=\sum_{\rho\in\mathbb{Z}}
    \Int^\rho(c_1,c_2)\]
is the number of geometric intersections between $c_1$ and $c_2$.
When $c_1=c_2$, the shift of the  identity map is the only extra base map which does not correspond to an intersection.
In particular, a boundary path $\za$ from $c_1$ to $c_2$ is considered as an oriented intersection from $c_1$ to $c_2$ with index $|\za|$, which corresponds to a common endpoint on the boundary if they collapse to arcs on $S$.
\end{theorem}

{Note that in the above theorem, when we consider morphisms between indecomposable objects, we require  the objects to be supported on curves in minimal position, that is, we minimize the number of intersections of the curves, see for example, \cite[Definition 7]{T22} for the precise definition of minimal position of curves.}

\begin{figure}[H]
\begin{center}
\begin{tikzpicture}[scale=1.5]
\draw (-2.5,-1.5) rectangle (2.5,1.5);
\draw[thick] (-.75,-.25) to [out=0,in=-90] (-.5,0) to [out=90,in=0] (-1,.5) to [out=180,in=90] (-1.25,.25);
\draw (-1.25,.25) to [out=180,in=90] (-1.5,0) to [out=-90,in=180] (-1,-.5) to [out=0,in=-90] (-.75,-.25);
\draw (1.25,-.25) to [out=0,in=-90] (1.5,0) to [out=90,in=0] (1,.5) to [out=180,in=90] (.75,.25);
\draw[thick] (.75,.25) to [out=180,in=90] (.5,0) to [out=-90,in=180] (1,-.5) to [out=0,in=-90] (1.25,-.25);
\draw[red,thick] (-.5,0) to (.5,0);
\draw[Green,thick] (-.567,.25) to [out=30,in=180] (1,.7) to [out=0,in=150] (2.5,0);
\draw[Green,thick] (-2.5,0) to [out=-30,in=180] (-1,-.7) to [out=0,in=-150] (.567,-.25);
\draw[blue,thick] (-.75,.433) to [out=60,in=-150] (2.5,1.5);
\draw[blue,thick] (-2.5,-1.5) to [out=30,in=-120] (.75,-.433);
\draw[orange,thick] (-1,.5) to [out=90,in=-135] (0,1.5);
\draw[orange,thick] (0,-1.5) to [out=45,in=-90] (1,-.5);
\draw[->] (-.5,0) arc (0:15:5mm) node[anchor=west] {$\za_1$};
\draw[->] (-.567,.25) arc (30:45:5mm) node[anchor=south west] {$\zb_1$};
\draw[->] (-.75,.433) arc (60:75:5mm) node[anchor=south] {$\zg_1$};
\draw[->] (.5,0) arc (180:195:5mm) node[anchor=east] {$\za_2$};
\draw[->] (.567,-.25) arc (210:225:5mm) node[anchor=north east] {$\zb_2$};
\draw[->] (.75,-.433) arc (240:255:5mm) node[anchor=north] {$\zg_2$};
\end{tikzpicture}
\end{center}
\caption{Example of a full formal arc system with four arcs on a genus one curve with two boundary components. The boundary paths are $\za_i,\zb_i,\zg_i$, $i\in\{1,2\}$, and their concatenations.}\label{fig:g1b2}
\end{figure}
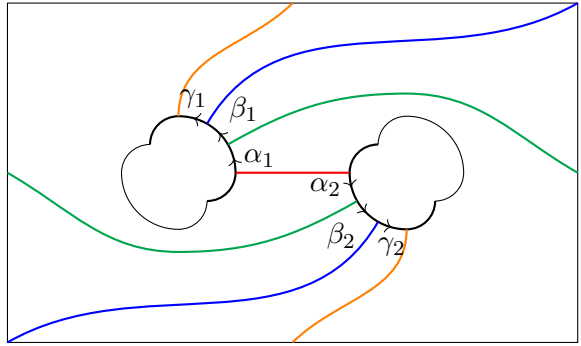

As an example, let $S$ be the genus one surface with two boundary components, i.e. $S$ is $S^1\times S^1$ with two open disks removed.
Choose a single marked point on each boundary component.
The real blow-up $\widehat S$ is depicted in Figure~\ref{fig:g1b2}, together with a possible choice of full formal arc system $\mathbb A$ consisting of four arcs which cut $S$ into a pair of pentagons.
The category $\mathcal F_{\mathbb A}$ thus has four objects and is described by the following quiver:
\begin{equation}\label{eq:quiver1}
{\begin{tikzpicture}[scale=0.4]
\def \radius {4cm}
    \draw [thick,->] (0.2,.2) -- (2.2,.2);
    \draw [thick,->] (.4+2.8,.2) -- (2.4+2.8,.2);
    \draw [thick,->] (.5+5.6,.2) -- (2.5+5.6,.2);

    \draw [thick,->] (0.2,-.2) -- (2.2,-.2);
    \draw [thick,->] (.4+2.8,-.2) -- (2.4+2.8,-.2);
    \draw [thick,->] (.5+5.6,-.2) -- (2.5+5.6,-.2);

    \draw[red] (-.3,0) node {$\bullet$};
    \draw[Green] (2.7,0) node {$\bullet$};
    \draw[blue] (5.6,0) node {$\bullet$};
    \draw[orange] (8.6,0) node {$\bullet$};

    \draw (1.2,0.7) node {$\za_1$};
    \draw (4.2,0.7) node {$\zb_1$};
    \draw (7.1,0.7) node {$\zg_1$};
    \draw (1.2,-0.7) node {$\za_2$};
    \draw (4.2,-0.7) node {$\zb_2$};
    \draw (7.1,-0.7) node {$\zg_2$};
\end{tikzpicture}}
\end{equation}
with relations stating that the composition of any arrow in the top row with an arrow in the bottom row vanishes.
The grading of the arrows depends on a choice of line field $\nu$ and grading of the arcs.
If we let $\nu$ be the constant foliation on the torus (e.g. given everywhere by the horizontal direction in Figure~\ref{fig:g1b2}), then we can choose the grading of the arcs so that all arrows are in degree zero.

\subsection{Exceptional sequences in Fukaya categories of surfaces}\label{subsec:exceptional}

We now recall some background on the theory of exceptional sequences in a general triangulated category $\calt$.

We call an object $X\in\calt$ \emph{exceptional} if
$\Hom_{\calt}(X,X[\neq0])=0$ and $\End_{\calt}(X)= \bf k$.
We call an (ordered) sequence $(X_1,\ldots,X_n)$ of exceptional objects in $\calt$ an \emph{exceptional sequence} if
$\Hom_{\calt}(X_i,X_j[\mathbb{Z}])=0,\ \mbox{ for }\ 1\le j<i\le n$. If in addition
$\thick_\calt(\bigoplus_{i=1}^nX_i)=\calt$ then we say that the sequence is a \emph{full exceptional sequence}.

Since the shift of an exceptional sequence is again an exceptional sequence, there is an action of  $\Z^n$  on the set of full exceptional sequences in $\mathcal T$ defined as follows
\[(\ell_1,\ldots,\ell_n)(X_1,\ldots,X_n):=(X_1[\ell_1],\ldots,X_n[\ell_n]).\]

Let $(X,Y)$ be an exceptional pair in $\calt$.
Define objects $R_{Y}X$ and $L_{X}Y$ in $\calt$ through the following triangles
\begin{equation}\label{equation:right exchange braid}
\xymatrix{
X\ar[r]&\coprod_{\ell\in\mathbb{Z}}D\Hom_{\calt}(X,Y[\ell])\otimes_kY[\ell]
\ar[r]&R_{Y}X[1]\ar[r]&X[1]
}\end{equation}
\begin{equation}\label{equation:left exchange braid}
\xymatrix{
Y[-1]\ar[r]& L_{X}Y[-1]\ar[r]&\coprod_{\ell\in\mathbb{Z}}\Hom_{\calt}(X[\ell],Y)\otimes_kX[\ell]\ar[r]&Y.
}\end{equation}
Then $(Y,R_{Y}X)$ and $(L_{X}Y,X)$ are again exceptional pairs in $\calt$ and we say that 
$R_{Y}X$ is the {\em{right mutation}} of $X$ at $Y$, and $L_{X}Y$ is the {\em{left mutation}} of $Y$ at $X$.

The mutations give rise to an action of $\mathfrak{B}_n$ on the set of full exceptional sequences on $\mathcal T$ as follows, see \cite{GR87}:
For a full exceptional sequence ${\mathbf X}:=(X_1,\ldots,X_n)$ and $1\le i<n$, set
$$\sigma_i{\mathbf X} := (X_1,\ldots,X_{i-1},X_{i+1},R_{X_{i+1}}X_{i},X_{i+2},\ldots,X_n)$$
$$\sigma_i^{-1}{\mathbf X} := (X_1,\ldots,X_{i-1},L_{X_i}X_{i+1},X_{i},X_{i+2},\ldots,X_n).$$

Recall from the introduction that by \textit{full exceptional sequence} we usually mean a full exceptional sequence up to shifts of the objects, i.e. an element of $\exc\calt/\Z^n$.
Crucially, the $\mathfrak{B}_n$-action descends to this quotient, which follows from the fact that the semidirect product $\mathbb{Z}^n\rtimes \mathfrak{B}_n$ acts.

%

In the case of graded marked surfaces $(S,M,\nu)$ where $M\subset \partial S$ and $\nu$ is such that the corresponding graded gentle algebras are concentrated in degree zero (i.e. gentle algebras in the classical sense), the first and third authors classified exceptional sequences in $\mathcal F(S,M,\nu)$ in terms of exceptional dissections on $(S,M)$, see~\cite{CS22}.
This result extends to the case of general $\nu$, see~\cite{CJS22}, where the result is proved inductively using recollements. However, the result for general $\nu$ can also be proved directly  along the lines of~\cite{CS22}. For completeness we include this proof here. 


\begin{theorem}\label{thm:cor-exp}
Let $(S,M,\nu)$ be a graded marked surface with $M\subset\partial S$, and let $\mathbb A$ be an exceptional dissection on $(S,M)$.
Then the sequence of objects in  $\mathcal F(S,M,\nu)$ corresponding to the arcs  in $\mathbb A$ is a full exceptional sequence in $\mathcal F(S,M,\nu)$.
Moreover, this gives rise to a bijection between exceptional dissections on $(S,M)$ and exceptional sequences in $\mathcal F(S,M,\nu)$.
\end{theorem}
\begin{proof}
Let $\mathbb A=(a_1,\ldots,a_n)$. For any choice of grading $\widetilde{a}_i$, denote by $X_{{a}_i}$ the associated object in $\mathcal F(S,M,\nu)$, where we abuse notation and view $a_i$ as its lift on $\widehat S$. Note that because each $a_i$ is an interval arc, there is no local system on it. Since there is no oriented self-intersection on $a_i$, $\dim \Hom^{\bullet}(X_{a_i},X_{a_i})=\Hom(X_{a_i},X_{a_i})=1$ by Theorem \ref{lemma:hom-int}, and thus $X_{a_i}$ is an exceptional object.
Furthermore, since there is no oriented intersection from $a_i$ to $a_j$ when $i > j$, thus again by Theorem \ref{lemma:hom-int}, $\Hom(X_{a_i},X_{a_j}[\mathbb Z])=0$. Therefore the sequence $\mathbf X=(X_{a_1},\ldots,X_{a_n})$ is an exceptional sequence in $\mathcal F(S,M,\nu)$.

To prove that any full exceptional sequence arises from an exceptional dissection, we start with describing indecomposable exceptional objects in $\mathcal F(S,M,\nu)$.
Let $X_{a}$ be an exceptional object associated to a graded curve $(I,a,\tilde{a})$ with a local system, then we claim that $a$ must be an embedded interval curve, and thus there is indeed no local system.

Assume $I=S^1$, then $X_a$ has a self-extension and can thus not be exceptional.
This self-extension is constructed as follows: Let $T\in\mathrm{GL}(d,\mathbf k)$ be the monodromy of the local system defining $X_a$ and consider the object $Y$ with the same underlying graded curve $(I,a,\tilde{a})$ and monodromy
\[
\begin{pmatrix} T & I_d \\ 0 & T \end{pmatrix}
\]
where $I_d$ is the identity matrix.
Then it follows from the construction of objects in $\mathcal F(S,M,\nu)$ from curves with local system that there is a non-split triangle 
$X_a \to Y \to X_a \to X_a[1]$.

Now let $I=[0,1]$ and assume that $a$ is not an embedding, then $p=a(x)=a(y)$, for some $x\neq y\in (0,1)$ or $p=a(0)=a(1)$. If $p=a(x)=a(y)$,  for some $x\neq y\in (0,1)$  then $p$ gives rise to two oriented intersections from $(I,a,\tilde{a})$ to itself with index $\rho$ and $1-\rho$ by equation \eqref{eq:index}, thus there are morphisms from $X_a$ to $X_a[\rho]$ and from $X_a$ to $X_a[1-\rho]$ by Theorem \ref{lemma:hom-int}. So $X_a$ is not exceptional, a contradiction.
If $p=a(0)=a(1)$ with index $\rho$, then $\rho\neq 0$ since $\End(X_a)$ is a field. But then $\Hom(X_a,X_a[\neq 0])$ is non-vanishing, and $X_a$ is still not exceptional. Therefore $a$ is an embedded interval curve.

Let $\mathbf X=(X_1,\ldots,X_n)$ be a full exceptional sequence in $\mathcal F(S,M,\nu)$, where each $X_i$ is associated to a graded curve $(I,a_i,\tilde{a}_i)$ with $I=[0,1]$. Then any two curves $a_i$ and $a_j$ have no interior intersection, since otherwise, Theorem \ref{lemma:hom-int} implies that both $\Hom(X_{a_i},X_{a_j}[\mathbb Z])$ and $\Hom(X_{a_j},X_{a_i}[\mathbb Z])$ are non-zero. Moreover, there is a partial order of the curves given by the boundary paths. This implies that the arcs in $\mathbb A$ cut the surface into subsurfaces each of which contains an (unmarked) boundary segment on its boundary.
Furthermore, $\mathbb{A}$ is maximal, that is, $\mathbb{A}$ has exactly $m+b+2g-2$ arcs, since $\mathbf X$ is full.
Then by \cite[Proposition 1.11]{APS19}, $\mathbb A$ is an admissible dissection and in particular, an exceptional dissection on $(S,M)$ by viewing each $a_i$ as an arc on $(S,M)$.
\end{proof}

\begin{remark}
More generally, one can ask for a classification of semi-orthogonal decompositions of $\mathcal F(S,M,\nu)$ in terms of the geometry of the surface. 
Such a classification is found, at least in the case where $\mathcal F(S,M,\nu)$ is the derived category of an (ungraded) gentle algebra, in the recent work~\cite{KS22}.
\end{remark}


The following theorem shows that the braid group actions on the set of exceptional dissections on $(S,M)$ and on the set of full exceptional sequences in $\mathcal F(S,M,\nu)$ are compatible with each other, where the case of trivial grading is proved in \cite{CS22}.

\begin{theorem}\label{thm:braid}
Let $\mathbf X=(X_1,\ldots,X_n)$ be a full exceptional sequence in $\mathcal F(S,M,\nu)$, which is supported over (a lift of) an exceptional dissection $\mathbb A=(a_1,\ldots,a_n)$ on $(S,M)$. Then $\sigma_i\mathbf X$ is supported over $\sigma_i(\mathbb A)$, and $\sigma^{-1}_i\mathbf X$ is supported over $\sigma^{-1}_i(\mathbb A)$.
\end{theorem}
\begin{proof}
We prove the case for $\sigma_i$, while the case for $\sigma^{-1}_i$ is similar.
Note that we only need to consider any exceptional pair. So we assume $n=2$ and $i=1$. In this case, $\sigma_1(\mathbf X)=(X_2,R_{X_{2}}X_{1})$, $\sigma_1(\mathbb A)=(a_2,R_{a_{2}}a_{1})$, and we have to show that $R_{X_{2}}X_{1}$ is supported over $R_{a_{2}}a_{1}$.
There are three cases depending on the number of boundary paths from $a_1$ to $a_2$. Here we abuse notation and view $a_i$ as an arc on $S$ as well as an arc on $\widehat S$.

The proof is clear when there is no boundary path between them.
Suppose now that there is exactly one boundary path $\za_1$ from $a_1$ to $a_2$ and denote the corresponding morphism also as 
$\za_1: X_1 \rightarrow X_2[|\za_1|]$. 
Then the object $R_{X_{2}}X_{1}$ arises from the triangle
\begin{equation}\label{equ:tri}
\xymatrix{
X_1\ar[r]^{\za_1}&X_2[|\za_1|]\ar[r]^{\za_2}&R_{X_{2}}X_{1}\ar[r]^{\za_3}&X_1[1].
}
\end{equation}
On the other hand, by the triangle structure of $\mathcal F(S,M,\nu)$ \cite{HKK17}, the arc associated to the third term in above triangle is a concatenation of $a_1$, $\za_1$ and $a_2$, see Figure \ref{figure:comp-mutation0}, which is exactly the arc $R_{a_{2}}a_{1}$ after collapsing it to $S$, see the first case in Figure \ref{figure:mutation of exceptional dissection}.
Note that here $a_2$ is considered up to isotopy with respect to the marked boundary, and the grading of $R_{a_{2}}a_{1}$ is the grading such that $|\za_2|=-|\za_1|$ for the boundary path $\za_2: a_2 \rightarrow R_{a_{2}}a_{1}$.

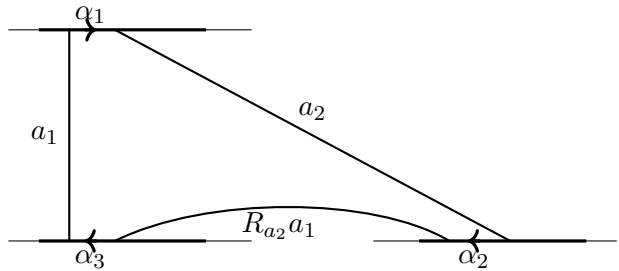
\begin{figure}[H]
\begin{center}
{\begin{tikzpicture}[scale=0.4]

\draw[-]  (-10,-3.5)to(-2,-3.5);
\draw[-]  (2,-3.5)to(10,-3.5);
\draw[-]  (-10,3.5)to(-2,3.5);

\draw[-,thick]  (-8,3.5)to(-8,-3.5);
\draw[-,thick]  (-6.5,3.5)to(6.5,-3.5);

\draw[very thick]  (-9,3.5)to(-3.5,3.5);
\draw[very thick]  (-9,-3.5)to(-3.5,-3.5);
\draw[very thick]  (9,-3.5)to(3.5,-3.5);

\draw[thick](-6.5,-3.5)..controls (-3.5,-2) and (2,-2)..(4.5,-3.5);

\draw (-8.8,0) node {$a_1$};
\draw (0,.8) node {$a_2$};
\draw (-1.1,-3) node {$R_{a_{2}}a_{1}$};

\draw (-7.3,4) node {$\za_1$};

\draw[->,very thick]  (-7.5,3.5)to(-7.1,3.5);

\draw (5.3,-4.1) node {$\za_2$};
\draw[->,very thick]  (5.5,-3.5)to(5.1,-3.5);

\draw (-7.3,-4.1) node {$\za_3$};
\draw[<-,very thick]  (-7.5,-3.5)to(-7.1,-3.5);
\end{tikzpicture}}
\end{center}

\begin{center}
\caption{The case when there is only one boundary path $\za_1$ from $a_1$ to $a_2$ in an exceptional pair $(a_1,a_2)$. Such a `triangle' on the surface corresponds to a distinguished triangle \eqref{equ:tri} in $\mathcal F(S,M,\nu)$ for some proper choice of grading.}\label{figure:comp-mutation0}
\end{center}
\end{figure}

Now assume that there are two boundary paths $\za_1$ and $\za_2$ from $a_1$ to $a_2$, see the Figure \ref{figure:comp-mutation1}. The morphisms associated to $\za_1$ and $\za_2$ give rise to two triangles in $\mathcal F(S,M,\nu)$, see the Figure \ref{figure:comp-mutation2}.
More precisely, let $a_3$ and $a_4$ be the concatenation of $a_1$, $\za_1$, $a_2$, and $a_1$, $\za_2$, $a_2$ respectively, and let $X_3$ and $X_4$ be the objects supported on $a_3$ and $a_4$ respectively. In particular, the grading on $a_3$ and $a_4$ are the grading such that $|\za_3|=|\za_4|=1$ for the boundary paths $\za_3: a_3\rightarrow a_1$ and $\za_4: a_4\rightarrow a_1$.

Then we complete the two triangles as the commutative diagram in Figure \ref{figure:comp-mutation2} by using the octahedral  axiom. In particular, (up to isotopy) the concatenation of $a_3$, $\za_3\za_2$, and $a_2$ coincides with the concatenation of $a_4$, $\za_4\za_1$, and $a_2$, which is denoted by $a_5$. Let $X_5$ be the object supported over $a_5$ with proper grading.

Note that the middle square in Figure \ref{figure:comp-mutation2} is a homotopy cartesian square, see for example \cite[Lemma 1.4.3]{N}.
Therefore we have the following distinguished triangle in $\mathcal F(S,M,\nu)$
\begin{equation}\label{equ:tri2}
X_1\s{(\za_1,\za_2)}\longrightarrow X_2[|\za_1|]\oplus X_2[|\za_2|]\s{\left(
                       \begin{smallmatrix}
                         -\za_5 \\
                         \za_7\\
                       \end{smallmatrix}
                     \right)}\longrightarrow X_5\longrightarrow {X_1}[1].
\end{equation}
Compare with \eqref{equation:right exchange braid}, we have $R_{X_2}X_1=X_5$. On the other hand, $a_5$ coincides with $R_{a_2}a_1$ after collapse to an arc on $(S,M)$. Thus we have proved the statment.

\begin{figure}[H]
\begin{center}
{\begin{tikzpicture}[scale=0.4]

\draw[-]  (-10,-3.5)to(-2,-3.5);
\draw[-]  (-10,3.5)to(-2,3.5);
\draw[-]  (2,-3.5)to(10,-3.5);
\draw[-]  (2,3.5)to(10,3.5);

\draw[-,thick]  (-8,3.5)to(-8,-3.5);
\draw[-,thick]  (8,3.5)to(8,-3.5);

\draw[-,thick]  (6.5,3.5)to(-6.5,-3.5);
\draw[-,thick]  (-6.5,3.5)to(6.5,-3.5);

\draw[very thick]  (-9,3.5)to(-3.5,3.5);
\draw[very thick]  (9,3.5)to(3.5,3.5);
\draw[very thick]  (-9,-3.5)to(-3.5,-3.5);
\draw[very thick]  (9,-3.5)to(3.5,-3.5);

\draw[thick](-4.5,3.5)..controls (-2,2) and (2,2)..(4.5,3.5);
\draw[thick](-4.5,-3.5)..controls (-2,-2) and (2,-2)..(4.5,-3.5);

\draw (-9,0) node {$a_2$};
\draw (9,0) node {$a_2$};
\draw (0,3) node {$a_3$};
\draw (0,-3) node {$a_4$};
\draw (5,2) node {$a_1$};
\draw (5,-2) node {$a_5$};

\draw (-7.3,4) node {$\za_5$};
\draw (-5.3,4) node {$\za_6$};
\draw[->,very thick]  (-7.5,3.5)to(-7.1,3.5);
\draw[->,very thick]  (-5.5,3.5)to(-5.1,3.5);

\draw (7.3,4) node {$\za_2$};
\draw (5.3,4) node {$\za_3$};
\draw[<-,very thick]  (7.5,3.5)to(7.1,3.5);
\draw[<-,very thick]  (5.5,3.5)to(5.1,3.5);

\draw (7.3,-4.1) node {$\za_7$};
\draw (5.3,-4.1) node {$\za_8$};
\draw[->,very thick]  (7.5,-3.5)to(7.1,-3.5);
\draw[->,very thick]  (5.5,-3.5)to(5.1,-3.5);

\draw (-7.3,-4.1) node {$\za_1$};
\draw (-5.3,-4.1) node {$\za_4$};
\draw[<-,very thick]  (-7.5,-3.5)to(-7.1,-3.5);
\draw[<-,very thick]  (-5.5,-3.5)to(-5.1,-3.5);
\end{tikzpicture}}
\end{center}

\begin{center}
\caption{The case when there are two boundary paths $\za_1$ and $\za_2$ from $a_1$ to $a_2$ in an exceptional pair $(a_1,a_2)$.}\label{figure:comp-mutation1}
\end{center}
\end{figure}
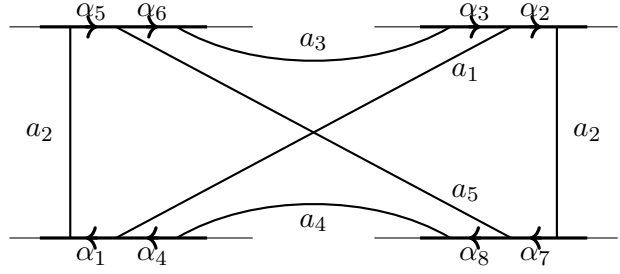

\begin{figure}[H]
\begin{center}
{\begin{tikzpicture}[scale=1.6]
\draw (1,0) node {$X_4$};
\draw (3,0) node {$X_4$};

\draw (1,1) node {$X_2[\za_2|]$};
\draw (3,1) node {$X_5$};
\draw (5,1) node {$X_3$};
\draw (-1,1) node {$X_3[-1]$};

\draw (1,2) node {$X_1$};
\draw (3,2) node {$X_2[\za_1|]$};
\draw (5,2) node {$X_3$};
\draw (-1,2) node {$X_3[-1]$};

\draw (1,3) node {$X_4[-1]$};
\draw (3,3) node {$X_4[-1]$};

\draw (2,2.12) node {$\za_1$};
\draw (2,1.12) node {$\za_7$};

\draw (4,2.12) node {$\za_5\za_6$};
\draw (4,1.12) node {$\za_6$};

\draw (1.32,.5) node {$\za_7\za_8$};
\draw (1.2,1.5) node {$\za_2$};
\draw (3.2,1.5) node {$\za_5$};
\draw (3.2,.5) node {$\za_8$};

\draw (1.2,2.5) node {$\za_4$};
\draw (0.1,2.12) node {$\za_3$};
\draw (3.3,2.5) node {$\za_4\za_1$};
\draw (0.1,1.2) node {$\za_3\za_2$};
\draw[-] (-1.05,1.2) -- (-1.05,1.8);
\draw[-] (-1,1.2) -- (-1,1.8);

\draw[-] (5.05,1.2) -- (5.05,1.8);
\draw[-] (5,1.2) -- (5,1.8);

\draw[-] (1.5,-0.02) -- (2.4,-0.02);
\draw[-] (1.5,0.03) -- (2.4,0.03);

\draw[-] (1.7,3.02) -- (2.3,3.02);
\draw[-] (1.7,2.97) -- (2.3,2.97);

\draw[->] (1.4,2) -- (2.5,2);
\draw[->] (1.4,1) -- (2.5,1);

\draw[->] (-0.3,2) -- (.5,2);
\draw[->] (-0.3,1) -- (.5,1);

\draw[->] (3.5,2) -- (4.5,2);
\draw[->] (3.5,1) -- (4.5,1);

\draw[->] (1,0.7) -- (1,0.2);
\draw[->] (1,1.7) -- (1,1.2);
\draw[->] (1,2.7) -- (1,2.2);

\draw[->] (3,0.7) -- (3,0.2);
\draw[->] (3,1.7) -- (3,1.2);
\draw[->] (3,2.7) -- (3,2.2);
\end{tikzpicture}}
\end{center}
\begin{center}
\caption{The four distinguished triangles in the commutative diagram correspond to the four `triangles' in Figure \ref{figure:comp-mutation1}.}\label{figure:comp-mutation2}
\end{center}
\end{figure}

\end{proof}

We have the following direct corollary.
\begin{corollary}
\label{cor:braid}
The group $\mathfrak{B}_n$ acts transitively on the set of full exceptional sequences in $\mathcal F(S,M,\nu)$ if and only if it acts transitively on the set of exceptional dissections on $(S,M)$.
\end{corollary}

\section{Examples with non-transitive braid group action}\label{section:counterexamples}

The set of exceptional dissections on a marked surface $(S,M)$, $\mathfrak S_1$, has an action of both the mapping class group of $S$, $\mathrm{Mod}(S)$, and the braid group.
{ Since the braid group acts transitively on the set $\mathcal H$ of Hurwitz systems of type $(S,M)$ (Theorem~\ref{thm:kluitmann}), and $\mathfrak S_1/\mathrm{Mod}(S)\cong\mathcal H$, it follows that the combined action of both groups on $\mathfrak S_1$ is transitive}. So in order to answer the question of transitivity of just the braid group action it is useful to understand the relation between the two groups.
This is the subject of \textit{Birman--Hilden theory}, where a common assumption is \textit{regularity} of the branched covering, i.e. the group of deck transformations acts transitively.
Since we also require our coverings to be simple, this forces us to restrict to branched \textit{double} covers of the disk.

Thus, let $\mathfrak{p}:S\to D$ be a double cover of the disk, branched over $p_1,\ldots,p_n$.
The \textit{mapping class group} of $S$, $\mathrm{Mod}(S)$, is the group of orientation preserving diffeomorphisms $S\to S$ fixing $\partial S$, modulo isotopy.
$\mathrm{Mod}(S)$ contains a subgroup $\mathrm{SMod}(S,\mathfrak p)$, the \textit{symmetric mapping class group}, of those mapping classes which contain a \textit{symmetric} diffeomorphism, i.e. obtained as a lift of a diffeomorphism $D\to D$.
On the other hand, $\mathrm{Mod}(D\setminus \{p_1,\ldots,p_n\})\cong \mathfrak{B}_n$.
Any symmetric diffeomorphism thus gives an element of $\mathfrak{B}_n$, but to have a well-defined map $\mathrm{SMod}(S,\mathfrak p)\to\mathfrak{B}_n$ we need to to know that two isotopic symmetric diffeomorphisms are isotopic through symmetric diffeomorphisms.
This question is answered positively by a theorem of Birman--Hilden~\cite{BH73}, see also \cite{FM12,MW21} for modern accounts.

\begin{theorem}[Birman--Hilden]
Let $\mathfrak p:S\to D$ be a double cover of the disk, branched over $n$ points, then $\mathrm{SMod}(S,\mathfrak p)\cong \mathfrak{B}_n$ via the map which sends a symmetric diffeomorphism to its corresponding map $D\to D$.
\end{theorem}

The map $\mathfrak{B}_n\to \mathrm{Mod}(S)$ sends a standard generator represented by a half-twist along an embedded curve $\alpha$ in $D$ connecting two branch points to the Dehn twist along the simple closed curve $\tilde{\alpha}$ in $S$ whose image under $\mathfrak p$ is $\alpha$.

When is $\mathrm{SMod}(S,\mathfrak p)=\mathrm{Mod}(S)$?
This question is studied for regular coverings of surfaces with boundary in work of Ghaswala and McLeay~\cite{GM20}.
As a special case we have the following fact.

\begin{theorem}[Ghaswala--McLeay]
Let $\mathfrak p:S\to D$ be a double cover of the disk, branched over $n$ points.
If $n\leq 3$, then $\mathrm{SMod}(S,\mathfrak p)=\mathrm{Mod}(S)$.
If $n\geq 4$, then $\mathrm{SMod}(S,\mathfrak p)$ is a subgroup of infinite index in $\mathrm{Mod}(S)$
\end{theorem}

Note that in the cases $n=1,2,3$, $S$ is the disk, the annulus, or the torus with one boundary component, respectively.
In the case $n=4$, $S$ is a torus with two boundary components.
Combining the above with our previous results, we obtain the following.

\begin{corollary}
Let $\mathfrak p:S\to D$ be a double cover of the disk, branched over $n$ points.
Then $\mathfrak{B}_n$ acts transitively on  exceptional dissections of $(S,p^{-1}(-1))$ iff $n\leq 3$.
\end{corollary}

\begin{proof}
We claim that $\mathrm{Mod}(S)$ acts freely on the set of  exceptional dissections (in general).
To see this, suppose $f\in\mathrm{Mod}(S)$ fixes some exceptional dissection $\mathbb A$.
Then $f$ fixes $M$ by definition and sends each arc to a homotopic arc.
After composing $f$ with a diffeomorphism isotopic to the identity, we can assume that $f$ fixes $\mathbb A$.
But since the graph formed by the arcs in $\mathbb A$ is a deformation retract of $S$, $f$ must be isotopic to the identity.

Note that if two exceptional dissections $\mathbb A_1$ and $\mathbb A_2$ on $(S,M)$ give rise to the same Hurwitz system, then there is an $f\in\mathrm{Mod}(S)$ which takes $\mathbb A_1$ to $\mathbb A_2$.
Under our assumptions on $\mathfrak p$, all exceptional dissections give rise to the same Hurwitz system, so $\mathrm{Mod}(S)$ acts transitively on the set of exceptional dissections.
Thus $\mathrm{SMod}(S,\mathfrak p)=\mathfrak B_n$ acts transitively iff $\mathrm{SMod}(S,\mathfrak p)=\mathrm{Mod}(S)$.
\end{proof}

Under the correspondence between exceptional dissections and exceptional sequences, and the compatibility of braid group actions on them, see Corollary \ref{cor:braid}, we thus obtain our counter-examples to the Bondal--Polishchuk conjecture.

\begin{theorem}
\label{thm:main1}
Let $(S,M,\nu)$ be a marked graded surface with $|M|=2$ and either one boundary component and $g(S)\geq 2$ or two boundary components and $g(S)\geq 1$.
Then the braid group action on the set of exceptional sequences in $\mathcal F(S,M,\nu)$ is not transitive but has infinitely many orbits.
\end{theorem}

In particular, for the Fukaya category $\mathcal F(S,M,\nu)$ with surface as in Figure \ref{fig:g1b2}, or equivalently, for the derived category of the gentle algebra with quiver \eqref{eq:quiver1}, the group action of $\mathbb{Z}^4\rtimes \mathfrak{B}_4$ on the set of full exceptional sequences is not transitive.
An example of two exceptional sequences which are not related by the action of  $\mathbb{Z}^4\rtimes \mathfrak{B}_4$ are given by the exceptional sequence associated to the exceptional dissection  in Figure \ref{fig:g1b2} and the exceptional dissection obtained after a Dehn twist of one of the two boundary components since such a Dehn twist is not an element of the symmetric mapping class group. 
Those cases in Theorem~\ref{thm:main1} where $\nu$ is such that the corresponding graded gentle algebra is concentrated in degree zero also appear in~\cite{D21}, providing examples of derived categories which are not silting connected.

A natural future direction is to analyze the case of $m$-to-$1$ coverings for $m>2$.
These seem to fall outside the scope of Birman--Hilden theory.

\begin{question}
For which marked surfaces $(S,M)$ is the action of the braid group on exceptional dissections (equivalently: branched coverings  with matching paths or framed Hurwitz systems) transitive?
\end{question}

Going beyond Fukaya categories of surfaces, one could explore whether the phenomenon of non-transitivity of the braid group action also appears in algebraic geometry.

\begin{question}
Are there counterexamples to Conjecture~\ref{conj:bp} where the triangulated category is of the type $D^b(\mathrm{Coh}(X))$ for some smooth projective variety $X$?
\end{question}

A possible general strategy for constructing counterexamples is suggested by the following observation.
{We say that a full exceptional sequence $\mathbf X=(X_1,\ldots,X_n)$ of a triangulated category $\mathcal C$ is \textit{preserved} by an autoequivalence $\Phi:\mathcal C\to\mathcal C$ if $\Phi(X_i)\cong X_i[m_i]$, for some $m_i\in\mathbb Z$, $i=1,\ldots,n$.
The following is immediate from the definition of mutation.
\begin{lemma}
If $\Phi$ preserves $\mathbf X$, then $\Phi$ preserves any exceptional sequence in the $\mathfrak B_n$-orbit of $\mathbf X$.
\end{lemma}
Thus, if one can find an autoequivalence $\Phi$, a full exceptional sequence $\mathbf X$ preserved by $\Phi$, and a full exceptional sequence $\mathbf X'$ not preserved by $\Phi$, then the braid group cannot act transitively on full exceptional sequences in $\mathcal C$.

In our counterexamples coming from branched double coverings $\mathfrak p:S\to D$ we can take $\Phi$ to be induced by the covering involution, $\mathbf X$ an exceptional dissection lifting a collection of matching paths in $D$, and $\mathbf X'=\varphi(\mathbf X)$ where $\varphi\in \mathrm{Mod}(S)\setminus \mathrm{SMod}(S,\mathfrak p)$. Here we need to assume that the grading $\nu$ is chosen so that it is invariant under the covering involution to ensure the existence of $\Phi$.}

\bibliographystyle{plain}
\bibliography{Chang-Haiden-Schroll}

\end{document}